\documentclass{cpamart1}     
\received{Month 200X}       
\volume{000}
\startingpage{1}                      

\authorheadline{N. Berger, A. Drewitz, A.F. Ram\'irez}
\titleheadline{Polynomial Ballisticity Conditions}

\usepackage{graphicx}
\usepackage{url}
\usepackage{color}

\usepackage[mathscr]{euscript}		

\usepackage{dsfont}

\usepackage{psfrag}			
\usepackage{hyperref}

\usepackage{times}

\usepackage{caption} 

\usepackage{anysize}

\usepackage{enumerate}
\usepackage{enumitem} 

 \usepackage{mathptmx}

\newtheorem{theorem}{Theorem}[section]
\newtheorem{lemma}[theorem]{Lemma}
\newtheorem{corollary}[theorem]{Corollary}
\newtheorem{proposition}[theorem]{Proposition}
\newtheorem{conjecture}[theorem]{Conjecture} 
\newtheorem{claim}[theorem]{Claim} 

\theoremstyle{definition}
\newtheorem{definition}[theorem]{Definition}
\theoremstyle{remark}
\newtheorem{remark}[theorem]{Remark}

\newcommand{\floor}[1]{{\lfloor #1 \rfloor}}
\newcommand{\Z}{\ensuremath{\mathbb{Z}}}
\newcommand{\N}{\ensuremath{\mathbb{N}}}
\newcommand{\R}{\ensuremath{\mathbb{R}}}

\newcommand{\E}{\ensuremath{\mathbb{E}}}
\renewcommand{\P}{\ensuremath{\mathbb{P}}}

\newcommand{\Pbox}{\ensuremath{(\mathcal P)}}
\newcommand{\Pasymp}{\ensuremath{(\mathcal P^*)}}

\renewcommand{\labelenumi}{(\alph{enumi})}

\newcommand{\degAss}{15d+5}

\newcommand{\largenessAss}{\exp \big \{ 100 + 4d (\ln \kappa)^2 \big\}}

\begin{document}                        


\title{Effective Polynomial Ballisticity Conditions\\
for Random Walk in Random Environment}


\author{Noam Berger}{The Hebrew University of Jerusalem}
\author{Alexander Drewitz}{ETH Z\"urich}
\author{Alejandro F. Ram\'irez}{Pontificia Universidad Cat\'{o}lica de Chile}





\begin{abstract}
The conditions $(T)_\gamma,$ $\gamma \in (0,1),$ which have been
introduced by Sznitman in 2002, have had a significant impact on research in random walk in random environment.
Among others, these conditions entail a ballistic behaviour as well as an invariance principle.
They require the stretched exponential decay of certain slab exit
probabilities for the random walk under the averaged measure and are asymptotic in nature.

The main goal of this paper is to show that in all relevant dimensions (i.e., $d \ge 2$),
in order to establish the conditions $(T)_\gamma$,
it is actually enough to check a corresponding condition $\Pbox{}$ of polynomial type.
In addition to only requiring an a priori weaker decay of the corresponding slab exit probabilities than $(T)_\gamma,$
another advantage of the condition $(\mathcal{P})$ is that it is effective in the sense that it can be checked
on finite boxes.

In particular, this extends the conjectured equivalence
of the conditions $(T)_\gamma,$ $\gamma \in (0,1),$ to all relevant dimensions.
\end{abstract}

\maketitle   






\section{Introduction  and statement of the main result Theorem \ref{thm:PolImpliesStretched}\\ (Polynomial decay is enough)}

\subsection{Introduction}

Random walk in random environment (RWRE) is a generalisation of simple random walk which serves as a model for
describing transport processes
in inhomogeneous media. Its study has originally been motivated
by its role as a toy model in the replication of DNA chains as well as by the investigation of phase transitions in alloys
(in particular the growth of crystals)
in the late 60's and early 70's of the last century, see e.g. Chernov \cite{Ch-67} and Temkin \cite{Te-72}.
In addition, the model is related to Anderson's tight-binding model for disordered electron systems
as well as to a deterministic motion among random scatterers (such as the Lorentz gas, see Sina{\u\i} \cite{Si-82b}).
Furthermore, it serves as a theoretical model exhibiting $1/f$-noise --- a phenomenon
frequently occurring in physics but hard to establish in theoretical models (see Marinari et al. \cite{MaPaRuWi-83}).

The model has attracted significant mathematical attention and has undergone a major
development during the last decades, establishing results on limiting velocities, as well as diffusive and
non-diffusive limiting laws, for example.

In particular, the model exhibits appealing phenomena not present in simple random walk. For instance,
the
question of whether RWRE exhibits diffusive behaviour has attracted considerable attention, and in fact in \cite{Si-82},
Sina{\u\i} showed that in a standard one-dimensional setting, RWRE $(X_n)$ has fluctuations of
scale $(\log n)^2$ only, in contrast to the diffusive scale $\sqrt{n};$
see Kesten, Kozlov and Spitzer \cite{KeKoSp-75} for further results in this direction
as well as Bricmont and Kupiainen \cite{BrKu-91} (and references therein)
also for a discussion of the multi-dimensional situation, where
understanding is still far from complete.

As another intriguing example, consider
for an element $l \in \mathbb{S}^{d-1}$ of the $d-1$-dimensional unit sphere
in $\R^d,$ the event $A_l := \{X_n \cdot l = \infty\}$ of \emph{transience in direction $l.$}
Then Kalikow's zero-one law states that $P_0(A_l \cup A_{-l}) \in \{0,1\}$ (cf. Kalikow \cite{Ka-81}, Sznitman and Zerner
\cite{SzZe-99}, as well as
Zerner and Merkl \cite{ZeMe-01}), where $P_0$ is the averaged probability defined in \eqref{eq:averagedProb} below;
however, in dimensions larger than two
it is not known whether $P_0(A_l) \notin \{0,1\}$ can occur or whether a corresponding zero-one law holds for $P_0(A_l)$ also.

Two of the main 
difficulties in investigating RWRE arise from
the fact that under the averaged measure, the walk
is not Markovian anymore as well as from its strongly non-self-adjoint character.
As a consequence, the power of spectral theoretic tools is of limited scope only.

In particular, coming back to the above-named difficulties in understanding the
 higher-dimensional situation, there is still no handy criterion
to  \emph{characterise} the situations in which the walk exhibits a non-vanishing limiting velocity (i.e. ballisticity).
However, the conditions $(T)_\gamma,$ $\gamma \in (0,1],$ introduced
by Sznitman in \cite{Sz-01} and \cite{Sz-02} have proven
to be useful in deriving many interesting results concerning the ballistic and diffusive behaviour of RWRE.

\subsection{Basic notation and known results}

In order to be more precise, we now give a short introduction to the model, thereby fixing some of the notation we employ.
We use $\Vert \cdot \Vert_1$ for the $1$-norm
and $\vert \cdot \vert$ for the absolute value.
By $\mathcal M_d$ we denote  the space of probability measures on the measurable space
${(\{e \in \Z^d : \Vert e \Vert_1 = 1\}, \mathcal A)}$
of canonical unit vectors, with $\mathcal A$ denoting the power set of ${\{e \in \Z^d : \Vert e \Vert_1 = 1\}},$
and we set $\Omega := (\mathcal M_d)^{\Z^d}.$
Elements of $\Omega$ will be referred to as {\it environments}, and for any ${\omega=(\omega(x,\cdot))_{x \in \Z^d} 
\in\Omega}$
one can consider a Markov chain $(X_n)_{n\in\N}$ with transition probabilities from $x$ to $x+e$
given by $\omega(x,e)$ if $\Vert e \Vert_1 = 1,$ and $0$ otherwise.
 We denote by $P_{x,\omega}$ the law of this Markov chain conditional on $\{X_0 = x\}.$
 By $\mathcal F$ we will denote the $\sigma$-algebra on $\mathcal M_d$ induced through the Borel-$\sigma$-algebra on $\R^{2d}$
(where elements of $\mathcal M_d$ are identified with the
elements of $\R^{2d}$ with non-negative entries summing up to $1$).
Furthermore, to account for the randomness of the environments, 
\begin{enumerate} [label={\bf (IID)}, ref={\bf (IID)}]
 \item \label{item:IIDassumption}
we assume $\P$ to be a probability measure on $(\Omega, \mathcal F^{\Z^d})$
such that
 the coordinates $(\omega(x,\cdot))_{x \in \Z^d}$ of the environment $\omega$
are independent identically distributed under $\P$.
\end{enumerate}
In this context, $\P$ is  called \emph{elliptic},
if $\P(\min_{\Vert e \Vert_1 =1} \omega(0,e)>0)=1,$ and it is called \emph{uniformly elliptic} if
there is a constant $\kappa>0$
such that 
$
\P(\min_{\Vert e \Vert_1 =1} \omega(0,e)\ge\kappa)=1.
$
For $\omega$ chosen accordingly to $\P,$ we
refer to $P_{x,\omega}$ as the {\it quenched law} of the RWRE starting from $x,$ and correspondingly
we define the {\it averaged} (or {\it annealed}) law of the RWRE by 
\begin{equation} \label{eq:averagedProb}
 P_x:=\int_\Omega P_{x,\omega} \,\P({\rm d}\omega).
\end{equation}
As mentioned above, by  $\mathbb S^{d-1}$ we denote the $(d-1)$-dimensional unit-sphere in $\R^d.$
Given a direction $l\in\mathbb S^{d-1},$ one refers to the RWRE as being  {\it transient
in the direction $l$} if
$$
P_{0}\Big(\lim_{n\to\infty} X_n\cdot l=\infty\Big)=1,
$$
and as being {\it ballistic in the direction $l$}
if $P_0$-a.s.
$$
\liminf_{n\to\infty}\frac{X_n\cdot l}{n}>0.
$$
In this context, the case $d=1$ has been resolved by Solomon \cite{So-75} who has given concise and useful
characterisations of the situations in which the walk
exhibits transient and ballistic behaviour, respectively.

\begin{theorem} [\cite{So-75}] \label{thm:Solomon}
Let $d=1$ and $\rho(0):= \omega(0,1)/\omega(0,-1).$ If $\E \ln \rho(0)$ is well-defined (possibly taking the values
$\pm \infty$),
then the events
$
{\{\lim X_n = \infty\},}
$
$
{\{\liminf X_n = -\infty, \liminf X_n = -\infty\},}
$
and
$
{\{ \lim X_n = -\infty \},}
$
have full $P_0$-probability according to whether 
$
{\E \ln \rho(0) > 0,}
$
$
{\E \ln \rho(0) = 0,}
$
and
$
{\E \ln \rho(0) < 0,}
$
respectively. Similarly, writing $
{v^+ := (1-\E\rho)/(1+\E \rho)}$ and ${v^- := (\E(\rho^{-1})-1)/(1+\E (\rho^{-1})),}$ the events
$
\{\lim X_n/n = v^+\},
$
$
\{\lim X_n/n = 0\},
$
and
$
\{\lim X_n/n = v^- \},
$
have full $P_0$-probability according to whether 
$
\E \rho(0) > 0,
$
$
\E \rho(0) = 0,
$
and
$
\E \rho(0) < 0,
$
respectively.
\end{theorem}
In particular, from this result one easily infers that
in $d=1,$ there exists uniformly elliptic RWRE that is transient but not ballistic to the right.
The picture is much more involved in dimensions larger than one, though.
In fact, there it has also been established that 
there exist elliptic RWRE in independent identically distributed environments which are
transient but not ballistic in a given direction, see for example Sabot and Tournier \cite{SaTo-11}.
On the other hand, however, even in the uniformly elliptic case there are still no useful 
characterisations of the situations in which RWRE
is transient or ballistic.  In order to facilitate redaction, we will abbreviate the condition of uniform ellipticity as  follows.
\begin{enumerate} [label={\bf (UE)}, ref={\bf (UE)}]
 \item \label{item:standardAssumptions1}
Let $\P$ be uniformly elliptic with ellipticity constant $\kappa > 0.$
\end{enumerate}
Then the following fundamental conjecture remains open.

\begin{conjecture} \label{conj:transBall}
Let $d \ge 2$ and
assume  \ref{item:IIDassumption} as well as \ref{item:standardAssumptions1} to hold.
Then RWRE which is transient with respect to all directions in an open subset of $\mathbb S^{d-1}$ is
necessarily ballistic.
\end{conjecture}
As hinted at above, some partial progress has been made towards the resolution of
this conjecture by studying RWRE satisfying
the conditions $(T)_\gamma.$
To rigorously formulate this condition,
let $L \ge 0$ and $l \in \mathbb S^{d-1}$ an element of the unit sphere. Then we write
\begin{equation} \label{eq:exitTimeOneDefX}
 H_L^l := \inf\{ n \in \N_0 : X_n \cdot l > L\}
\end{equation}
for the first entrance time of $(X_n)$ into the half-space $ \{x \in \Z^d : x \cdot l > L\},$
and where ${\N_0 = \{0, 1, 2, \ldots\}.}$

\begin{definition}[\cite{Sz-02}] \label{def:Tgamma}
Let $\gamma\in (0,1]$ and $l\in\mathbb S^{d-1}$.
We say that \emph{condition $(T)_\gamma$} is satisfied with respect to $l$
(written $(T)_\gamma|l$ or $(T)_\gamma$) if for each $l'$ in a neighbourhood of $l$ and
each $b>0$ one has that

\begin{align*}
\limsup_{L \to \infty} L^{-\gamma} \ln P_0 \big( H_L^{l'} > H_{bL}^{-l'} \big) < 0.
\end{align*}
We say that \emph{condition $(T')$} is satisfied with respect to $l$ (written
$(T')|l$ or $(T')$), if for each $\gamma\in (0,1)$, condition $(T)_\gamma|l$
is fulfilled.
\end{definition}
In the following we will shortly explain the importance of the conditions $(T)_\gamma.$

It is known that in dimensions $d\ge 2$ and assuming \ref{item:IIDassumption} and \ref{item:standardAssumptions1}, the validity of the condition
$(T')$ already implies the existence of a deterministic $v \in \R^d \backslash\{0\}$ such that
$P_0$-a.s. $\lim_{n \to \infty} \frac{X_n}{n} = v,$ as well as an invariance principle for the
RWRE so that under the averaged law $P_0$,
$$
B^n_\cdot := \frac{X_{\floor{\cdot n}}-\floor{\cdot n}v}{n}
$$
converges in distribution to a Brownian motion in the Skorokhod space
$D([0,\infty), \R^d)$ as $n \to \infty;$ see for instance Theorem 4.1 in Sznitman \cite{Sz-04}
for further details. 
Recently, the condition $(T')$ has also been used to obtain further knowledge about large deviations for RWRE, see
e.g.  Berger \cite{Be-09}.

While $(T)_\gamma$ a priori is a stronger condition the larger $\gamma$ is,
it has been shown in Sznitman \cite{Sz-02} by a detour along the so-called \emph{effective criterion} that for $d \geq 2,$
assuming \ref{item:IIDassumption} and \ref{item:standardAssumptions1},
the conditions
$(T)_\gamma$ are equivalent for all $\gamma \in (\frac12,1).$ This equivalence has been further improved in 
Drewitz and Ram\'irez \cite{DrRa-09b} to
all $\gamma \in (\gamma_d,1)$ for some constant $\gamma_d \in (0.366,0.388).$ For dimensions larger or equal to four,
it has been established in Drewitz and Ram\'irez \cite{DrRa-10}
by different methods
that the conditions $(T)_\gamma$ are actually equivalent for all $\gamma \in (0,1).$
Note that in \cite{Sz-04} it has been conjectured by Sznitman that for any $d \ge 2$ fixed, the conditions $(T)_\gamma$ are equivalent for all $\gamma 
\in (0,1],$ and we are making another step towards the resolution of this conjecture as a corollary (see Corollary \ref{cor:equiv}) to our main result
Theorem \ref{thm:PolImpliesStretched}.

\subsection{Main result}

The principal goal of this paper is to significantly weaken the condition that has to be checked in order to establish $(T'),$
and hence ballisticity. For this purpose
we set
\begin{equation} \label{eq:c0Def}
c_0 := \largenessAss < \infty,
\end{equation}
and introduce the following definition.
%

\begin{definition} \label{def:PCond}
Let $M > 0$ and $l \in \mathbb S^{d-1}.$
 We say that {\em condition $\Pasymp_M \vert l$} is satisfied 
with respect to $l$
(also written $\Pasymp_M$ or $\Pasymp$ at times)
if 
the following holds: 
For all
$b > 0$ and all $l' \in \mathbb S^{d-1}$ in some neighbourhood of $l,$ one has that
%
%
\begin{equation} \label{eq:PCond}	
\limsup_{L \to \infty} L^{M} P_0 \big( H_{bL}^{-l'} < H_{L}^{l'} \big) = 0,
\end{equation}
\end{definition}

\begin{remark}
\begin{enumerate}
\item
 In fact, throughout the whole paper the condition $\Pasymp$ of Definition \ref{def:PCond}
can be replaced by the weaker condition
$\Pbox$ given in Definition \ref{def:Pbox}. The condition $\Pbox$ of that definition is also more adequate to check in examples since it is effective
in the sense that it can be verified on finite boxes already.
%

Note, however, that condition $\Pasymp$ is better suited to illustrate the relations with Sznitman's conditions $(T)_\gamma$.

The reason for giving Definition \ref{def:PCond} of $\Pasymp$ here instead of $\Pbox$
 is that the latter requires quite some notation which will only be introduced
later on. As will be shown in Lemma \ref{lem:asympBox}, condition $\Pasymp$ implies $\Pbox.$ Until the formal introduction of 
condition $\Pbox$ in Definition \ref{def:Pbox}, to facilitate reading we will state both assumptions, condition $\Pbox$
 and condition $\Pasymp,$ in results.

\item 
Note that if $\Pasymp_M \vert l$ holds, then
$\Pasymp_M \vert l'$ holds for all directions $l'$ in a neighbourhood of $l$ in $\mathbb S^{d-1}$  also.
\end{enumerate}
\end{remark}

It is straightforward that condition $(T') \vert l$ implies $\Pasymp_M \vert l$ for any $M \in (0,\infty).$
The main result of the paper states that the converse is true also, provided that $M$ is large enough.
\begin{theorem}[Polynomial decay is enough] \label{thm:PolImpliesStretched}
Assume \ref{item:IIDassumption} and \ref{item:standardAssumptions1} to be fulfilled.
 Let $l \in \mathbb{S}^{d-1}$ and assume that $\Pasymp_M \vert l$ or $\Pbox_M \vert l$ holds for some $M> \degAss.$
Then $(T') \vert l$ holds.
\end{theorem}

The importance of this result also stems from the multitude of results that so far have been known to hold
under the condition $(T')$ only. Using Theorem \ref{thm:PolImpliesStretched},
it is now sufficient to establish the polynomial decay of the exit probabilities corresponding to $\Pbox_M$ or $\Pasymp_M$ instead
of the a priori stronger stretched exponential decay
of $(T)_\gamma.$
 In particular, Theorem \ref{thm:PolImpliesStretched} can be seen as a major step towards proving
Conjecture \ref{conj:transBall}.

In addition, in contrast to the conditions $(T)_\gamma,$ the condition $\Pbox_M$ can be checked on
finite boxes (without a detour along an analogue to the \emph{effective criterion} of \cite{Sz-02}),
which emphasises its effective character, cf. Definition \ref{def:Pbox}.

Furthermore, combining Theorem \ref{thm:PolImpliesStretched}
with the above remark that $(T') \vert l$ implies $\Pasymp_M \vert l,$
we directly obtain the following corollary.
\begin{corollary} \label{cor:equiv}
Assume \ref{item:IIDassumption} and \ref{item:standardAssumptions1} to be fulfilled.
Then for any $l \in \mathbb S^{d-1},$ the conditions $(T)_\gamma \vert l,$
$\gamma \in (0,1),$ are equivalent.
\end{corollary}

\subsection{Some further notation} \label{subsec:Not}
For $k \in \N,$ we define the
canonical left shift
\begin{equation} \label{eq:shiftDef}
\theta_k: \R^\N \ni (x_n)_{n \in \N} \mapsto (x_{k+n})_{n \in \N} \in \R^\N.
\end{equation}
Throughout the rest of the paper, $C$ will denote differing
strictly positive and finite constants. Their precise values
may change from one side of an inequality to the other; however, in particular, they do not depend on the parameter $L$ that will
be employed frequently in the paper.
If we want to refer to constants that may depend on the dimension and the ellipticity constant $\kappa$
but otherwise are absolute, we put indices as in ${c}_3$ for example.

For a $\Z^d$-valued discrete time stochastic process $(Y_n)$ 
and $A \subset \Z^d$ we define the entrance time into $A$ as
\begin{equation} \label{eq:entranceTime}
H_A^Y:=H_A(Y):= \inf \big \{ n \in \N_0 : Y_n \in A \big \},
\end{equation}
and for singletons $A=\{z\}$ we denote $H_z (Y) := H_{\{z\}}(Y).$
Furthermore, for $l \in \mathbb{S}^{d-1}$ and $L \in \R,$ in accordance with \eqref{eq:exitTimeOneDefX}, we define the entrance time
\begin{equation} \label{eq:exitTimeOneDef}
 H_L^l(Y) := \inf \big\{ n \in \N_0 : Y_n \cdot l > L \big\},
\end{equation}
into the half-space $\{ x \in \Z^d : x \cdot l > L\},$
with the usual convention that $\inf \emptyset := \infty.$
Similarly, for a set $A$ the exit time is defined as
$$
T_A^Y:=T_A(Y):= \inf \big \{ n \in \N_0 : Y_n \notin A \big \},
$$
When referring to the canonical RWRE $(X_n)$ that we will be dealing with, then for the sake of simplicity we will often
omit $X$ as an argument of the entrance and exit times.

For any subset $A \subset \Z^d$ its (outer) boundary $\partial A$ is defined to be
\begin{equation} \label{eq:bdryDef}
 \partial A := \big \{x \in \Z^d \backslash A : \exists y \in A \text{ such that } \Vert x-y \Vert_1 = 1 \big \}.
\end{equation}

For $l\in\mathbb S^{d-1}$
we will
use the notation
\begin{equation} \label{eq:pi}
\pi_l: \R^d \ni x \mapsto (x \cdot l) \,l  \in \R^d
\end{equation}
to denote the orthogonal projection on the space $\{\lambda l : \lambda \in \R\}$
as well as
\begin{equation}\label{eq:piBot}
\pi_{l^\bot} : \R^d \ni x \mapsto x - \pi_l(x) \in \R^d
\end{equation}
for the projection on the corresponding orthogonal subspace.

Now for a generic $l_1 := l \in \mathbb S^{d-1}$
we choose and fix for the remaining part of this article
$l_2, l_3, \ldots, l_d$ arbitrarily such that in combination with $l_1$ these vectors form
 an orthonormal basis of $\R^d.$
Furthermore, for $L > e^e$ define
$$
\mathcal{D}_L^l := \Big \{ x \in \Z^d : -L \leq x  \cdot l \leq 10L, \vert x \cdot l_j \vert 
\leq  \frac{L^{3} \ln \ln L}{\ln L} \; \; \forall \; 2 \le j \le d \Big \}
$$
as well as its frontal boundary part
$$
\partial_+ \mathcal{D}_L^l := \Big \{ x \in \partial \mathcal D_L^l : x \cdot l > 10L, \vert x \cdot l_j \vert 
\leq  \frac{L^{3} \ln \ln L}{\ln L} \; \; \forall \; 2 \le j \le d \Big \}.
$$
We introduce the following condition \eqref{eq:superPolDecay} for further reference. Its validity under $\Pbox_M \vert l'$
and $\Pasymp_M \vert l',$ respectively, will be
the content of Proposition \ref{prop:superPolDecay}.
\begin{align}
\begin{split}\label{eq:superPolDecay}
&\text{For }l' \in \mathbb S^{d-1} \text{ one has }
P_0\Big (H_{\partial \mathcal D_L^l} < H_{\partial_+ \mathcal D_L^l} \Big )
\leq \exp \Big \{ -L^{\frac{(1+o(1))\ln 2}{\ln \ln L}} \Big\},
\quad  \text{as } L \to \infty.
\end{split}
\end{align}

\begin{remark} \label{rem:gammaLDef}
If \eqref{eq:superPolDecay} holds,  in correspondence to condition $(T)_\gamma$ of Definition \ref{def:Tgamma}, we write
\begin{equation*} 
\gamma_L := \frac{\ln 2}{\ln \ln L}
\end{equation*}
to denote
the \emph{effective} $\gamma.$
\end{remark}

\begin{definition}
If
 $\eqref{eq:superPolDecay}$ holds,  then we say that
condition $(T)_{\gamma_L} \vert l$
is fulfilled.
\end{definition}

\section{Proof of Theorem \ref{thm:PolImpliesStretched} (Polynomial decay is enough)} \label{sec:Strategy}

In 
Subsection \ref{subse:auxRes} we state two auxiliary results that
will be helpful in the proof of Theorem \ref{thm:PolImpliesStretched} in Subsection \ref{subsec:proofMainThm}.

\subsection{Auxiliary results (Propositions \ref{prop:superPolDecay} and 
\ref{prop:quenchedExit})} \label{subse:auxRes}

In this subsection we state two results that play a key role in proving Theorem \ref{thm:PolImpliesStretched}.
Their proofs will be the subject of Sections \ref{sec:renormalisation} and \ref{sec:proofPropQuenchedExit}.

\begin{proposition}[Sharpened averaged exit estimates] \label{prop:superPolDecay}
Assume \ref{item:IIDassumption} and \ref{item:standardAssumptions1} to be fulfilled.
Let $M>\degAss,$ $l\in\mathbb S^{d-1}$ and assume that condition
  $\Pasymp_M|l$ or $\Pbox_M \vert l$ 
is satisfied. Then 
$(T)_{\gamma_L} \vert l$ holds.
\end{proposition}
The previous proposition will be proven in Section \ref{sec:renormalisation}.

To be able to formulate the second essential ingredient we have to recall the effective criterion
which has been introduced in \cite{Sz-02} and can be seen as an analogue to the conditions of Solomon (cf. Theorem
\ref{thm:Solomon})
in higher dimensions. 

For positive numbers $L,$ $L'$ and $\widetilde{L}$ as well as a space rotation $R$
around the origin we define the 
\begin{align*}
\text{{\it box specification} }{\mathcal{B}}(R, L, L', \widetilde{L})
\text{ as the box }
B:= \big \{x\in\mathbb Z^d:x\in R((-L,L') \times (-\widetilde{L}, \widetilde{L})^{d-1}) \big\}.
\end{align*}
 Furthermore, let
$$
\rho_{\mathcal{B}}(\omega) := \frac{P_{0,\omega} \big({H_{\partial B}} \not= H_{\partial_+ B} \big )}{P_{0,\omega} \big (H_{\partial B} =H_{\partial_+ B} \big)}.
$$
Here,
$$
\partial_+ B := \Big\{x \in \partial B : R(e_1) \cdot x \geq L', \vert R(e_j) \cdot x \vert < \widetilde{L} \; \forall j \in \{2, \dots, d\} \Big\}.
$$
We will sometimes write $\rho$ instead of $\rho_{\mathcal{B}}$ if the box we refer to is clear from the context
and use $\hat{R}$ to label any rotation mapping $e_1$ to $\hat{v}.$ 
Given $l\in\mathbb{S}^{d-1}$, the {\it effective
  criterion with respect to $l$} is satisfied if
for some $L > c_2$ and $\widetilde{L} \in [3\sqrt{d}, L^3),$ we have that
\begin{equation} \label{effectiveCritInf}
\inf_{{\mathcal{B}}, a} \Big\{ c_3 \Big(\ln \frac{1}{\kappa} \Big)^{3(d-1)} \widetilde{L}^{d-1} L^{3(d-1)+1} \E \rho_{\mathcal{B}}^a \Big\} < 1.
\end{equation}
Here, when taking the infimum,
$a$ runs over $[0,1]$ while ${\mathcal{B}}$ runs over the 
\begin{equation} \label{eq:spex}
\text{box specifications }
{\mathcal{B}}(R, L-2, L+2, \widetilde{L}) \text{ with $R$ a rotation around the origin such that $R(e_1) = l.$}
\end{equation}
Furthermore, $c_2$ and $c_3$ are dimension dependent constants.

The effective criterion is of significant importance due to its equivalence to $(T')$ (cf. Theorem \ref{Sz24}) and the fact that
it can be checked on finite boxes (in comparison to $(T')$ which is asymptotic in nature).

\begin{theorem}[\cite{Sz-02}] \label{Sz24}
Assume \ref{item:IIDassumption} and \ref{item:standardAssumptions1} to be fulfilled.
 For each $l \in \mathbb{S}^{d-1}$  the following
conditions are equivalent.
\begin{enumerate}
\item The effective criterion with respect  to $l$ is satisfied.

\item
$(T')|l$   is satisfied.

\end{enumerate}
\end{theorem}

We can now formulate the second key-ingredient for our proof of Theorem \ref{thm:PolImpliesStretched}.

\begin{proposition}[Atypical quenched exit estimates] \label{prop:quenchedExit}
Assume \ref{item:IIDassumption} and \ref{item:standardAssumptions1} to hold.
Furthermore, let
$(T)_{\gamma_L} \vert l$
be fulfilled.
Then, for 
$\epsilon(L):=\frac{1}{(\ln\ln L)^2},
$ and each function $\beta:(0,\infty)\to (0,\infty),$ 
one has that
\begin{equation} \label{eq:quenchedExit}
 \P \Big(P_{0,\omega}  \big(H_{\partial B}= H_{\partial_+ B} \big) \leq \frac{1}{2} \exp \big\{ -c_1 L^{\beta(L)} \big\}
 \Big)
\le 5^d \frac{e}{\lceil L^{\beta(L)-\epsilon(L)}/5^d \rceil!},
\end{equation}
where $B$ is a box specification as in \eqref{eq:spex} with $\widetilde L = L^3-1,$ and
\begin{equation} \label{eq:c4Def}
c_1 := - 2d \ln \kappa > 1.
\end{equation}
\end{proposition}

The proof of this result is the subject of Section \ref{sec:proofPropQuenchedExit}.

\subsection{Proof of Theorem \ref{thm:PolImpliesStretched} (assuming Propositions \ref{prop:superPolDecay} and 
\ref{prop:quenchedExit})} \label{subsec:proofMainThm}

Our proof of Theorem \ref{thm:PolImpliesStretched} goes along establishing the effective criterion and in the following
we will give some lemmas that will prove useful in this.

For that purpose, we define the quantities
\begin{equation} \label{eq:gammaDef}
 \beta_1(L)  := \frac{\gamma_L}{2} = \frac{\ln 2}{2 \ln \ln L},
\end{equation}
\begin{equation} \label{eq:alphaDef}
\alpha(L) := \frac{\gamma_L}{3} =\frac{\ln 2}{3 \ln \ln L},
\end{equation}
\begin{equation} \label{eq:aDef}
a:=L^{-\alpha(L)}
\end{equation}
and write $\rho$ for $\rho_{\mathcal B}$ with
some arbitrary box specification of \eqref{eq:spex} with $\widetilde L = L^3-1.$ We split
$
\E \rho^a
$
according to 
\begin{align}
\label{decomp}
 \E \rho^a = \mathcal{E}_0 + \sum_{j=1}^{n-1} \mathcal{E}_j+\mathcal{E}_n,
\end{align}
where 
\begin{equation*} 
n:=n(L):=\Big \lceil \frac{4(1-\gamma_L/2)}{\gamma_L} \Big \rceil + 1,
\end{equation*}
$$
\mathcal{E}_0 := \E \Big( \rho^a, P_{0,\omega} \big({H_{\partial B}} =
H_{\partial_+ B} \big) > \frac12 \exp \big\{-c_1 L^{\beta_1} \big\} \Big),
$$
$$
\mathcal{E}_j := 
\E \Big( \rho^a, \frac12 \exp\big\{ -c_1 L^{\beta_{j+1}} \big\}
 < P_{0,\omega}({H_{\partial B}} = H_{\partial_+ B}) \leq \frac12 \exp \big \{- c_1 L^{\beta_j} \big \} \Big)
$$
for $j \in \{1, \ldots, n-1\},$
and
$$
\mathcal{E}_n := \E \Big( \rho^a, P_{0,\omega}\big({H_{\partial B}} = H_{\partial_+ B}\big) \le \frac12 
\exp \big \{- c_1 L^{\beta_n} \big\} \Big),
$$
with parameters 
\begin{equation} \label{eq:betaDef}
\beta_j(L) := \beta_1(L) + (j-1) \frac{\gamma_L}{4},
\end{equation}
for $2\le j\le n(L);$  for the sake of brevity we may sometimes omit the dependence on $L$  of the parameters
if that does not cause any confusion.
Furthermore, in order to verify that equality \eqref{decomp} is indeed true,
note that due to the uniform ellipticity assumption \ref{item:standardAssumptions1}
and the choice of $c_1$ (cf. \eqref{eq:c4Def}),
one has for $\P$-a.a. $\omega$ that
$$
P_{0,\omega} \big( {H_{\partial B}} = H_{\partial_+ B} \big) > e^{-c_1 L},
$$
as well as that
\begin{equation*} 
\beta_{n} >1.
\end{equation*}
To bound $\mathcal{E}_0$ we employ the following lemma.
\begin{lemma}
\label{I} Let $(T)_{\gamma_L} \vert l$ be fulfilled. Then
\begin{equation} 
\nonumber
\mathcal{E}_0
\leq \exp \Big\{ c_1 L^{\gamma_L/6} - L^{(1+o(1))\gamma_L/2} \Big\},
\end{equation}
as $L \to \infty.$
\end{lemma}
\begin{proof}
Jensen's inequality yields
\begin{align*}
\mathcal{E}_0
\leq 2\exp\big\{ c_1 L^{\beta_1-\alpha} \big\}
P_0 \big ({H_{\partial B}} \not= H_{\partial_+ B} \big)^a.
\end{align*}
Using \eqref{eq:alphaDef} and \eqref{eq:gammaDef}, 
 in combination with  $(T)_{\gamma_L}$ we obtain the desired result.
\end{proof}

To deal with the middle summand in the right-hand side of (\ref{decomp}), we
use the following lemma.

\begin{lemma} \label{II}
Assume \ref{item:IIDassumption} and \ref{item:standardAssumptions1} to hold and
 let $(T)_{\gamma_L} \vert l$ be fulfilled. Then for all $L$ large enough we have uniformly in
${j \in \{1, \ldots, n-1\}}$ that
\begin{equation*} 
\mathcal{E}_j
\leq 2 \cdot 5^d \exp \big\{c_1 L^{\beta_{j+1}-\alpha}\big\} \frac{e}{ \lceil L^{\beta_j-\epsilon(L)}/5^d \rceil!}.
\end{equation*}
\end{lemma}
\begin{proof}
For $j \in \{1, \ldots, n-1\}$ we obtain the estimate
\begin{align} \label{(II)Est}
\mathcal{E}_j\leq 2\exp \big\{c_1 L^{\beta_{j+1} - \alpha}\big\}
\P \Big( P_{0,\omega} \big({H_{\partial B}} = H_{\partial_+ B} \big) \leq \frac12  \exp \big\{ -c_1 L^{\beta_{j}} \big\} \Big).
\end{align}
Due to Proposition
\ref{prop:quenchedExit},
the probability on the right-hand side of (\ref{(II)Est}) can be estimated from
above by
$$
 5^d \frac{e}{\lceil L^{\beta_j-\epsilon(L)}/5^d \rceil!},
$$
which finishes the proof.
\end{proof}
With respect to the term $\mathcal{E}_n$ in (\ref{decomp}) we note that it vanishes
due to the choice of $c_1.$

\begin{proof}[Proof of Theorem \ref{thm:PolImpliesStretched}]
It follows from  Lemmas \ref{I}, \ref{II}, the choice of parameters in \eqref{eq:gammaDef} to \eqref{eq:aDef}
and \eqref{eq:betaDef}, and the fact that $\mathcal E_n$ vanishes, that 
for $L$ large enough, (\ref{decomp}) can be bounded from above by 
\begin{align*}
\exp \Big\{ &c_1 L^{\gamma_L/6} - L^{(1+o(1))\gamma_L/2} \Big\} \\
&+  
2 \cdot 5^d n(L) \max_{1 \le j \le n(L)-1} \Big( \exp \big\{ c_1 L^{\beta_{j+1} - \gamma_L/3} \big\} 
\frac{e}{\lceil L^{\beta_j - \epsilon(L)}/5^d\rceil !} \Big).
%
\end{align*}
Thus, we see that for our choice of parameters, \eqref{decomp} tends to zero faster than any polynomial in $L.$
Hence, \eqref{effectiveCritInf} holds for $L$ large enough and the effective criterion is fulfilled, which in combination with
 Theorem \ref{Sz24} then yields the desired result.
\end{proof}

\section{Proof of Proposition \ref{prop:superPolDecay} (Sharpened averaged exit estimates)} \label{sec:renormalisation}

In this section we prove Proposition \ref{prop:superPolDecay} which has been employed in the proof of 
Theorem \ref{thm:PolImpliesStretched} in Section \ref{sec:Strategy}.

\subsection{Renormalisation step and condition $\Pbox$}
In this subsection we describe a renormalisation scheme that will finally lead to the proof of Proposition \ref{prop:superPolDecay}.


Let $N_0$ be an even integer larger than $c_0,$ where we recall that the latter has been defined in 
\eqref{eq:c0Def}. For $k \in \N_0,$ define recursively the scales
\begin{equation} \label{eq:NkDef}
N_{k+1} := 3(N_0+k)^2 N_k.
\end{equation}
We introduce for $k \in \N_0$ and $x \in \Z^d$ the boxes
\begin{equation}  \label{eq:boxes}
B(x,k) := \Big \{ y \in \Z^d : -\frac{N_k}{2} < (y-x) \cdot l < N_k,  \vert (y-x)\cdot l_j \vert
< 25N_k^{3} \; \; \forall \, 2 \le j \le d \Big\},
\end{equation}
as well as their frontal parts
\begin{equation} \label{eq:Btilde}
 \widetilde B(x,k) := \left\{ y \in \Z^d : N_k- N_{k-1} \leq (y-x) \cdot
l < N_k,  \vert (y-x)\cdot l_j \vert
< N_k^{3} \; \; \forall \, 2 \le j \le d \right\},
\end{equation}
with the convention that $N_{-1}:=2N_0/3.$
Furthermore, we define
\begin{equation} \label{eq:partialPlusDef}
\partial_+ B(x,k) := \Big\{y \in \partial B(x,k) : (y-x) \cdot l \geq N_k, \vert (y-x)\cdot l_j \vert
< 25N_k^{3} \; \; \forall \, 2 \le j \le d \Big\}.
\end{equation}
We will call $\widetilde{B}(x,k)$ its {\it middle frontal part} of $B(x,k)$
and sometimes refer to $\partial_+ B(x,k)$ as the {\it frontal boundary part} of $B(x,k).$
%
Furthermore, for $n_1, n_2 \in \N$ 
we fix the subset
\begin{equation} \label{eq:latticeType}
\mathcal L_{n_1, n_2} := \left\{ \left \lfloor \sum_{k=1}^d j_k l_k \right \rfloor \, : \, j_1 \in n_1 \Z, \, j_2, \ldots, j_d \in n_2 \Z \right\},
\end{equation}
of $\Z^d,$  where 
$\floor{\cdot}$ is understood coordinatewise. The letter $\mathcal L$ is chosen in order to be reminiscent of ``lattice''; one should however notice that 
$\mathcal L_{n_1, n_2}$ is only close to being a lattice in some sense.
We refer to the elements of
$$\mathfrak{B}_k := \left\{ B(x,k) : x \in
\mathcal{L}_{N_{k-1} -2, 2N^{3}_k -2} \right\}$$ 
as {\it boxes of scale $k$.}

To simplify notation, throughout we will denote a typical box of scale $k$ by $B_k$, and
its middle frontal part by $\widetilde B_k$. The reader should clearly distinguish such boxes from the box
configurations introduced around \eqref{effectiveCritInf}.

\begin{remark} \label{rem:coverProp}
For later reference note that the middle frontal parts of any scale cover $\Z^d,$ i.e., for any $k \ge 0$ one has 
$$
\bigcup_{B(x,k) \in \mathfrak{B}_k} \widetilde{B}(x,k) = \Z^d.
$$
\end{remark}

We can now introduce the condition $\Pbox{}.$

\begin{definition} \label{def:Pbox}
Let $l \in \mathbb S^{d-1}$ and $M > 0.$
We say that $\Pbox_M \vert l$ is fulfilled if
\begin{align} \label{eq:singleBound}
 \sup_{x \in \widetilde{B}_0} P_x \big(H_{\partial B_0} \ne H_{\partial_+ B_0} \big) < N_0^{- M}
\end{align}
holds for some $N_0 \ge c_0.$
\begin{remark}
In particular, note that although there does not occur any explicit dependence  on $l$ in  Definition \ref{def:Pbox}, it comes into
play via the displays \eqref{eq:boxes} to \eqref{eq:partialPlusDef}. Also, the very choice of $x$ for the box $B_0 = B(x,0)$ is irrelevant due to the translation invariance of $\P$ with respect to lattice
shifts.
\end{remark}


\end{definition}

\begin{lemma} \label{lem:asympBox}
For $M \in (0,\infty)$ and $l \in \mathbb S^{d-1},$ condition $\Pasymp_M \vert l$ implies $\Pbox_M \vert l.$
\end{lemma}
\begin{remark} \label{rem:Pbox}
Due to this result, from now on we will only refer to condition $\Pbox.$
\end{remark}

\begin{proof}[Proof of Lemma \ref{lem:asympBox}]
If $\Pasymp_M \vert l$ holds, we can choose vectors $l_1', \dots, l_d'$ in a neighbourhood of $l$ and $b > 0$ small enough such that
\begin{enumerate}

\item
one has
\begin{equation} \label{eq:wellBehaved}
\bigcap_{j=1}^d \Big \{ H^{l_j'}_{N_0}  < H^{-l_j'}_{bN_0} \Big\} \cap \big \{X_0 \in \widetilde B_0 \big \} 
\subseteq \big \{H_{\partial B_0} = H_{\partial_+ B_0} \big \}
\end{equation}
(see Figure \ref{fig:boxAndFront} for an illustration also);

\item 
inequality
\eqref{eq:PCond} with $l$ replaced by $l_1', \ldots, l_d',$
holds true.
\end{enumerate}
Then for arbitrary 
$x \in \widetilde B_0,$ we have using \eqref{eq:wellBehaved} that 
\begin{align*}
P_x \Big( H_{\partial B_0} \ne H_{\partial_+ B_0} \Big) \le \sum_{j=1}^d 
P_x \Big ( H^{-l_j'}_{bN_0} < H^{l_j'}_{N_0} \Big ),
\end{align*}
and for $N_0$ large enough, the last sum can be bounded from above by $N_0^{-M}$ due to \eqref{eq:PCond}.
This implies \eqref{eq:singleBound} and hence finishes the proof.
\end{proof}

\begin{figure}[h]
\begin{center}
\large

\psfrag{tildeBk}{$\widetilde{B}_0$}
\psfrag{Bk}{$B_0$}
\psfrag{hatv}{$l$}
\psfrag{partialPlus}{$\partial_+ B_0$}
\psfrag{lN}[1]{$l'N_0$}
\psfrag{blN}{$bl'N_0$}

  \scalebox{.9}{\includegraphics{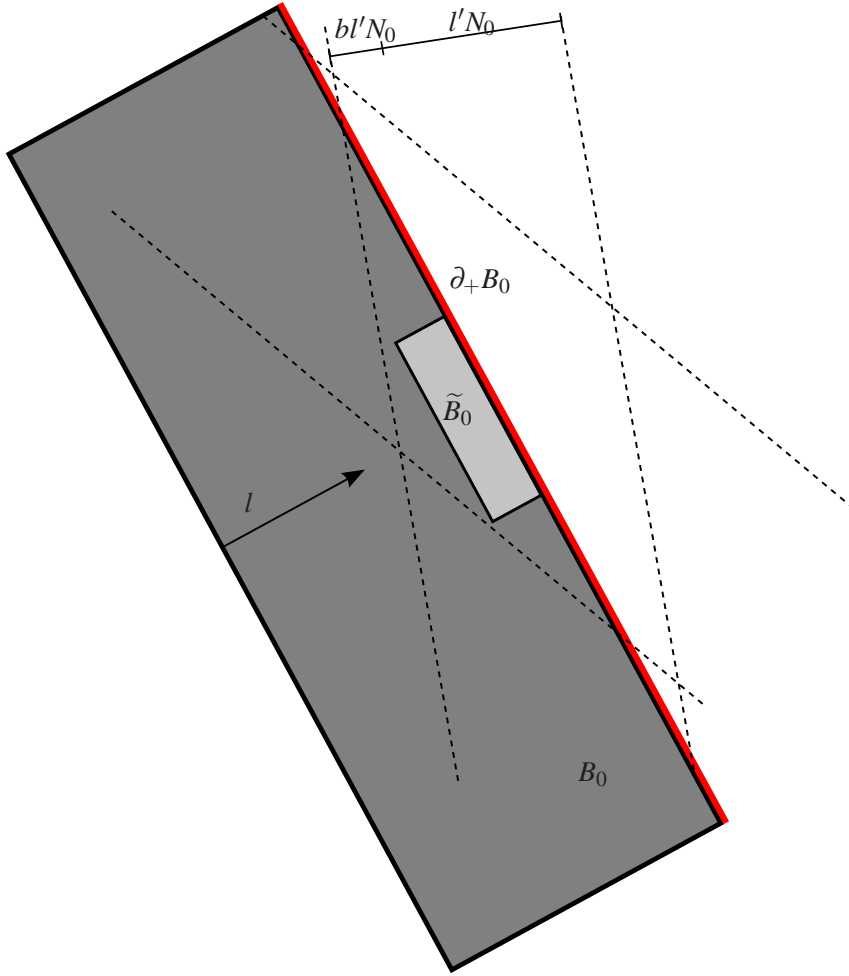}}
\caption{\sl 
A box $B_0$ and the middle frontal part $\widetilde B_0$. These boxes are much wider than they are long.
The dashed lines illustrate the slabs of width $(1+b)N_0$ in direction $l'$ from the definition of $\Pasymp_M \vert l$ (or similarly in the proof of Lemma \ref{lem:asympBox} also)
shifted by some $x \in \widetilde{B}_0.$
}
\label{fig:boxAndFront}
\end{center}
\end{figure}  

\begin{definition}\label{def:goodbad} {\bf (Good boxes)}.
We say that a box $B_0 \in \mathfrak{B}_0$ is \emph{good (with respect to $\omega \in \Omega$)} 
if
\begin{equation}
\inf_{x \in \widetilde B_0} P_{x,\omega} \big(H_{\partial
B_0} = H_{\partial_+ B_0} \big) \geq 1-N_0^{-5}.
\end{equation}
Otherwise, we say that the box is \emph{bad}.
For $k\ge 1$ we say that a box $B_k \in \mathfrak{B}_k$ is
\emph{good (with respect to $\omega \in \Omega$)}, if there is a box $Q_{k-1} \in \mathfrak{B}_{k-1}$
of scale $k-1$  such that every
box $B_{k-1} \in \mathfrak{B}_{k-1}$ of scale $k-1$ satisfying $B_{k-1} \cap Q_{k-1} = \emptyset$
and $B_{k-1} \cap B_k \ne \emptyset,$ is good (with respect to $\omega \in \Omega$).
Otherwise, we say that the box $B_k$ is \emph{bad}.
\end{definition}

We show that for $M$ large enough, condition $\Pbox_M \vert l$ implies that
boxes of scale $k$ are bad with a $\P$-probability decaying at least doubly-exponentially in $k,$ and start with the
case $k=0.$
\begin{lemma} \label{lem:goodAnnealedExitProb}
Let $l\in\mathbb S^{d-1}$
and assume that $\Pbox_M|l$ holds.
 Then  for all $B_0\in\mathfrak B_0$ and $N_0\ge c_0,$ one has that
\begin{equation}
\nonumber
 \P(B_0 \text{ is good})\geq 1- 2^{d-1} N_0^{3d+4-M}.
\end{equation}
\end{lemma}
\begin{proof} Note that
\begin{equation}
\label{eq:unionBound}
 \P(B_0 \text{ is  bad})\leq 
\sum_{x\in\widetilde B_0}\P\Big(P_{x,\omega} \big(H_{\partial B_0}\ne H_{\partial_+ B_0} \big)
\ge N_0^{-5}\Big).
\end{equation}
Now by Markov's inequality we have for $x \in \widetilde{B}_0$ that
\begin{equation} \label{eq:Cheby}
\P\Big(P_{x,\omega} \big(H_{\partial
B_0} \ne H_{\partial_+ B_0} \big )\ge N_0^{-5} \Big) \le
N_0^{5}  \sup_{x \in \widetilde{B}_0} P_x \big (H_{\partial B_0} \ne H_{\partial_+ B_0} \big).
\end{equation}
In combination with \eqref{eq:unionBound} and \eqref{eq:Cheby}, 
assumption $\Pbox_M$  implies that
$$
 \P(B_0 \text{ is  bad})\le \vert \widetilde{B}_0 \vert  N_0^{5} N_0^{-M}
\leq 2^{d-1}N_0^{3d+4-M},
$$
where we used that $\vert \widetilde B_0 \vert \le 2^{d-1}N_0^{3d-1}$ due to \eqref{eq:Btilde},
which thus finishes the proof.
\end{proof}

Next we treat the case of a general $k \in \N_0.$

\begin{proposition}\label{prop:likelygood}
Let $l \in \mathbb{S}^{d-1},$ $M > \degAss,$ and assume that $\Pbox_{M} \vert l$ is satisfied.
Then for $N_0 \ge c_0$ one has for all $k \in \N_0$ and all $B_k \in \mathfrak{B}_k$ that
\begin{align}  \label{eq:ProbGoodBox}
\P(B_k \text{ is good}) \geq 1-\exp \{-2^k\}.
\end{align}
\end{proposition}
\begin{proof}
For the sake of simplicity, denote $p_k := \P(B_k \text{ is good})$ for $B_k$ as in the assumptions
as well as $q_k := 1-p_k.$
We
will prove by induction that
\begin{equation} \label{eq:indStep}
 \P(B_k \text{ is good}) \geq 1-\exp \{-c_k'2^k\}
\end{equation}
for all $k \in \N_0,$ where 
$$
c_k' := \Big( 12d+\frac{2}{3} \Big ) \ln N_0 - \sum_{j=1}^k \frac{\ln ((90(j+N_0))^{12d})}{2^j}.
$$
Afterwards, we will then show that for $N_0$ as in the assumptions one has that $\inf_{k \ge 0} c_k' \ge 1,$ which will finish the proof.

\emph{Induction start:}

Lemma \ref{lem:goodAnnealedExitProb} yields that
\begin{align*}
\P (B_0 \text{ is good})  \ge 1 - \exp \big\{ -(12d + 2/3) \ln (N_0) \big\},
\end{align*}
which in particular implies that \eqref{eq:indStep} holds for $k=0.$

\emph{Induction step:}

Assume $k \geq 1$ and that \eqref{eq:indStep} holds for $k-1.$
Let  $q_{k-1}=\exp\{-c_{k-1}'2^{k-1} \}$ and
let $B_{k-1,1},B_{k-1,2},\ldots,B_{k-1,m_k}$ be all the boxes in $\mathfrak B_{k-1}$ that intersect $B_k$.
By Definition \ref{def:goodbad}, if each two bad boxes among $B_{k-1,1},B_{k-1,2},\ldots,B_{k-1,m_k}$
have a non-empty intersection, then the box $B_k$ is good.

Therefore, all  we need to do is to upper bound the probability that there exist two non-intersecting
boxes among $B_{k-1,1},B_{k-1,2},\ldots,B_{k-1,m_k}$ which are bad. By the union bound and assumption
\ref{item:IIDassumption} we get that
\[
q_k\leq {m_k \choose 2}q_{k-1}^2.
\]
Noting that for all $k \ge 1$ we have $m_k\leq (30\cdot 3(k+N_0))^{6d},$
the induction hypothesis yields
\begin{eqnarray*}
q_k\leq (90(k+N_0))^{12d}\big( \exp \{-c'_{k-1}2^{k-1} \}\big)^2
=\exp \big\{ \ln ((90(k+N_0))^{12d}) - c'_{k-1}2^{k}\big\},
\end{eqnarray*}
so $c'_k = c'_{k-1}-\frac{\ln ((90(k+N_0))^{12d})}{2^k}$ and hence inductively for every $k$,
\begin{equation} \label{eq:ckIneqI}
 c'_k\geq c'_0 - \sum_{j=1}^\infty \frac{\ln ((90(j+N_0))^{12d})}{2^j}.
\end{equation}
The sum obviously converges, but we need to compare it with the value of $c'_0$.
By Lemma \ref{lem:goodAnnealedExitProb} and since $M \ge \degAss,$ we deduce that for $N_0$ as in the assumptions,
\begin{equation} \label{eq:c0IneqII}
 c'_0\geq (12d+2/3)\ln N_0.
\end{equation}
To estimate the sum, we note that due to 
$
\ln (1+x+y) \leq \ln(1+x) + \ln (1+y),
$
for $x,y \geq 0,$ we have
for $N_0$ as in the assumptions  that
\begin{align} \label{eq:ckIneqIII}
\sum_{j=1}^\infty \frac{\ln \big((90(j+N_0))^{12d} \big)}{2^j}
& \leq 12d \Big( \ln 90 + \sum_{j=1}^\infty \frac{\ln (j+N_0)}{2^j} \Big)
\leq (12d+1/2)  \ln N_0.
\end{align}
Therefore, in combination with
\eqref{eq:ckIneqI} to \eqref{eq:ckIneqIII} it follows that
\[
c'_k \geq c:=c'_0 - \sum_{j=1}^\infty \frac{\ln \big ( (90(j+N_0))^{12d} \big)}{2^j}\geq \frac16 \ln N_0 >1,
\]
for every $k$, and where the last inequality holds since $N_0 \ge c_0,$ where $c_0$ as in \eqref{eq:c0Def}.
 Hence, ${q_k\leq \exp\{-2^k \}}$
as desired.
\end{proof}

Next we show that with high probability, a walker starting in the middle frontal part of a good box
leaves it through the frontal boundary part. 
For this purpose, we define the back boundary part
$$
\partial_- B(x,k) := \Big \{y \in \partial B(x,k) : (y-x) \cdot l \le -N_k/2,
\vert (y-x)\cdot l_j \vert
< 25N_k^{3} \; \; \forall \, 2 \le j \le d \Big\}
$$
as well as the side boundary part
$$
\partial_s B(x,k) := \partial B(x,k) \backslash \big( \partial_+B(x,k) \cup \partial_-B(x,k) \big).
$$

\begin{proposition} \label{prop:GoodBoxRightExit}
Let $N_0 \ge c_0,$ with $c_0$ as in \eqref{eq:c0Def}.
Then there is a constant $c_4>0$ such that for each $k\in \N_0$ and
$B_k\in\mathfrak B_k$ which is good with respect to $\omega,$ one has

\begin{equation*} 
\sup_{x\in \widetilde{B}_k}P_{x,\omega} \big( H_{\partial B_k}\neq H_{\partial_+ B_k} \big)
\le \exp\left\{-c_4N_k\right\}.
\end{equation*}

\end{proposition}
\begin{proof} For the sake of simplicity, we assume without loss of generality that $B_k=B(0,k).$
Using that
$$
P_{x,\omega}  \big (H_{\partial B_k}\ne H_{\partial_+ B_k} \big)
\le P_{x,\omega} \big( H_{\partial B_k}=H_{\partial_{s} B_k} \big)
+ P_{x,\omega} \big( H_{\partial B_k}=H_{\partial_- B_k}\big),
$$
we split the proof into two parts.
We will first prove that
\begin{equation}\label{eq:sideExitProbEst}
\sup_{x\in \widetilde{B}_k} P_{x,\omega}\big(H_{\partial B_k}=H_{\partial_{s} B_k} \big) \le
\exp\{-cN_k\}
\end{equation}
and then that
\begin{equation}\label{eq:back}
\sup_{x\in \widetilde B_k}P_{x,\omega} \big( H_{\partial B_k}=H_{\partial_- B_k} \big) \le
\exp\{-cN_k\},
\end{equation}
for some constant $c_4:= c >0,$
which will finish the proof.
To prove \eqref{eq:sideExitProbEst} and \eqref{eq:back} we proceed as follows:
Define the sequences $(c_k')_{k \in \N_0}$ and $(c_k'')_{k \in \N_0}$
via
\begin{equation*} 
c_k':= \frac{5 \ln N_0}{N_0} - \sum_{j=1}^k \frac{ \ln 30(N_0+j)^4}{N_{j-1}}
\end{equation*}
and
\begin{align*} 
c_k''
=\frac{5 \ln N_0}{N_0} - \sum_{j=1}^k \left(\frac{4\ln 3}{N_{j-1}} +
\frac{2N_{j-1}  + 3 N_{j-1} \ln 3  + 2 \ln 6 - 3c_1 N_{j-1} \ln \kappa}{N_j} \right).
\end{align*}
We will show that
\begin{equation} \label{eq:quenchedSideExitEst}
\sup_{x\in \widetilde{B}_k} P_{x,\omega} \big( H_{\partial B_k}=H_{\partial_{s} B_k} \big)\le
\exp\{-c_k'N_k\}
\end{equation}
and
\begin{equation} \label{eq:backSideExitEst}
\sup_{x\in \widetilde{B}_k} P_{x,\omega} \big(H_{\partial B_k}=H_{\partial_{-} B_k} \big)\le
\exp\{-c_k''N_k\}
\end{equation}
hold true for all $k \in \N_0.$
Displays \eqref{eq:sideExitProbEst} and \eqref{eq:back} will then follow since 
$$
c:= \inf_{k \in \N_0}(c_k') \wedge \inf_{k \in \N_0}(c_k'')
> 0
$$
for $N_0 \ge c_0.$

\emph{Induction start:}

For $k=0,$ displays  \eqref{eq:quenchedSideExitEst} and \eqref{eq:backSideExitEst}
follow from the definition of a good box at scale $0.$ 

\emph{Induction step:}

Now assume that \eqref{eq:quenchedSideExitEst} and \eqref{eq:backSideExitEst} hold
for scale $k-1$ where $k \geq 1.$

\emph{Proof of \eqref{eq:quenchedSideExitEst} for $k:$}

Let $\tau_1$ be the first time that the random walk leaves one of the
boxes of scale $k-1$ whose middle frontal parts contain the starting point $x \in \widetilde B_k$.
Define recursively for $n\ge 1$ the stopping time $\tau_{n+1}$
as the first time that the random walk leaves the box of scale $k-1$
whose middle frontal part contains the point $X_{\tau_n}$ (this is where we take advantage of
 Remark \ref{rem:coverProp}). If there is more than
one such box, then we choose one arbitrarily. We now consider the sequence
defined by 
\begin{equation}\label{eq:defresc}
Y_0:=x \quad \text{ and } \quad Y_n:=X_{\tau_n}, \quad \text{ for } n \in \N,
\end{equation}
and call $(Y_n)$ the {\em rescaled random walk}.

Since the box $B_k$ is good, we know that there exists  a
box $Q_{k-1} \in \mathfrak{B}_{k-1}$ such that every box of scale $k-1,$ intersecting
$B_k$ but not $Q_{k-1},$ is good.
With this notation we define
\begin{align*} 
 \begin{split}
 \mathfrak{B}_{Q_{k-1}} := \Big \{&B_{k-1} \in\mathfrak B_{k-1}: B_{k-1} \cap B_k \ne \emptyset \text{ and }
\text{ there exists } 
z\in B_{k-1} \text{ such that } \vee_{l^\bot}(z, Q_{k-1}) < 1 \Big\};
\end{split}
\end{align*}
here, for a set $A \subset \Z^d$ and $x \in \Z^d$ we use the notation 
$$
\vee_{l^\bot}(x,A) := \max_{2\le j \le d} \inf_{y \in A} \vert (x-y) \cdot l_j \vert.
$$
I.e.,
$\mathfrak{B}_{Q_{k-1}}$ is the collection of
boxes of scale $k-1$ which,
orthogonally to $l,$ are very close to 
$Q_{k-1}.$
For later reference we also introduce
 Next, we define
$$
\mathcal{Q}_{k-1} := \bigcup_{B_{k-1} \in \mathfrak{B}_{Q_{k-1}}} B_{k-1},
$$
 see Figure \ref{fig:mathcala}.
\begin{figure}[h]
\begin{center}
\huge

\psfrag{hatv}{$l$}
\psfrag{Bk}{$B_k$}
\psfrag{Qk-1}{$\mathcal{Q}_{k-1}$}
\psfrag{calQk-1}{$Q_{k-1}$}

 \scalebox{.6}{\includegraphics{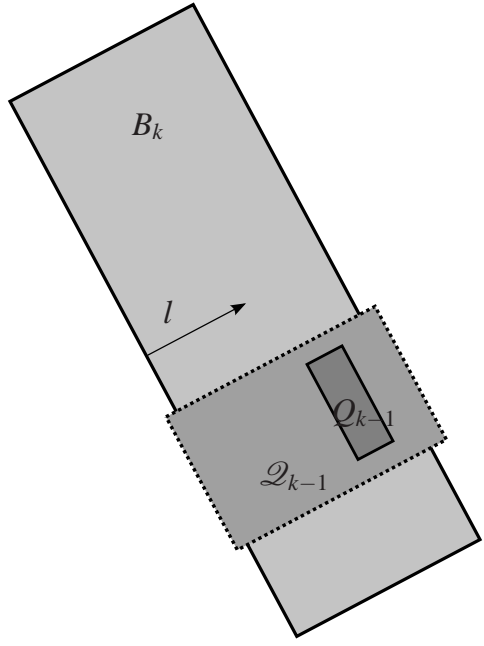}}
\caption{\sl 
The bad box $Q_{k-1}$ and its superset $\mathcal Q_{k-1}$ inside the box $B_k$.
}
\label{fig:mathcala}
\end{center}
\end{figure} 

Let now $m_1$ be the first time that the random walk $(Y_n)$
is at a distance larger than $7N_{k}^{3}$ from $\mathcal  Q_{k-1}$ and
from the sides $\partial_{s}B_k$ of the box $B_k$, in the sense that
$$
m_1:=\inf \Big \{n \in \N_0 : \vee_{l^\bot}\big(Y_n,{\mathcal Q}_{k-1}\big)\ge 7N_k^{3} \quad \text{and} \quad
\vee_{l^\bot} \big(Y_n,\partial_{s}B_k \big)\ge 7N_k^{3} \Big \}.
$$
Define $m_2$ as the first time that $(Y_n)$ exits
the box $B_k$ so that
$$
m_2:=\inf\{n \in \N_0, Y_n\notin B_k\}.
$$

Furthermore, we define 
\begin{equation*}
m_3 := \inf\{n>m_1: Y_n\in {\mathcal Q}_{k-1}\}\leq\infty
\end{equation*}
and
note that  on the event $\big\{ H_{\partial B_k}=H_{\partial_{s}B_k} \big\},$
we have that $P_{x,\omega}$-a.s., 
\begin{equation*}
m_1<m_2<\infty.
\end{equation*}
Therefore, $m':=(m_2 \wedge m_3) \circ \theta_{m_1}$ is well-defined on that event
and writing
$$
J_k:=2 \cdot \frac{3N_k /2}{N_{k-2}+1},
$$
one has $P_{x,\omega}(\cdot \, \vert \, H_{\partial B_k}=H_{\partial_{s}B_k})$-a.s. that
\begin{equation} \label{eq:mDifferences}
m' \geq \frac{7N_{k}^{3}}{30N_{k-1}^{3}} \ge J_k \frac{N_k}{20N_{k-1}}.
\end{equation}
Next observe that, again on $\{ H_{\partial B_k}=H_{\partial_{s}B_k}\},$ if $(Y_n)$ starts from some $y \in B_k$ such that
$$
\min \big\{ \vee_{l^\bot} (y, \partial_s B_k), \vee_{l^\bot} (y, \mathcal Q_{k-1}) \big\} \ge 30N_{k-1}^3 J_k,
$$
and if it consecutively leaves $J_k$ boxes of scale $k-1,$
then at least 
$20$ (we could do significantly better here --- however, since this is sufficient for our purposes, we leave it this way for the sake of simplicity)
such boxes must have been left through the frontal parts of their boundaries.\footnote{Indeed,
for each box of scale $k-1$ that $(Y_n)$ (feels and) leaves through its frontal boundary part, the position of the
walk gains at least $N_{k-2}-1$ units in direction $l.$ The ``most efficient'' way to decrease its position in $l$-direction is to leave a box
of scale $k-1$ through its back or side boundary part, which would decrease the $l$-coordinate of its position by at most $N_{k-1}+1.$ Therefore, 
if $(Y_n)$ has not left $B_k$ through its frontal or back boundary part within $J_k$ steps, then it must have left at least
$$
\frac{J_k - \frac32 N_k/N_{k-1}}{N_{k-1}+1} \ge 20
$$
boxes of scale $k-1$ not through their frontal boundary part.
}
Thus,
we have
that
$$
P_{y,\omega} \big(Y_j \in B_k \; \forall j \in \{1, \ldots, J_k\}  \big)
\le J_k^{20} \big( \exp\{-c_{k-1}' N_{k-1}\} \big)^{20}.
$$
This in combination with \eqref{eq:mDifferences} and the Markov property applied at times which are multiples of $J_k$
supplies us with
\begin{align*}
  P_{x,\omega} (H_{\partial B_k}=H_{\partial_{s}B_k}) 
&= P_{x,\omega} \Big(m' \ge J_k \frac{N_k}{20N_{k-1}}, H_{\partial B_k}=H_{\partial_{s}B_k} \Big)\\
&\le \big(\exp\{ -20c_{k-1}'N_{k-1} + 20 \ln J_k \} \big)^{\frac{N_k}{20N_{k-1}}} \le \exp \{ -c_{k}' N_k \}.
\end{align*}
This completes the proof of \eqref{eq:quenchedSideExitEst} for $k.$

\medskip
\emph{Proof of \eqref{eq:backSideExitEst} for $k:$}

In addition to the induction assumption that \eqref{eq:quenchedSideExitEst} and \eqref{eq:backSideExitEst} hold for scale $k-1,$ we can now assume that
\eqref{eq:quenchedSideExitEst} holds for scale  $k$ also.

The proof is based on a comparison of the ``$l$-coordinate'' of the rescaled random walk $(Y_n)$ with a one-dimensional walk
with drift.

%
%

Assume the statement holds for $k-1$ with $k\geq 1.$
Let $B_k \in \mathfrak{B}_k$ be a good box of scale $k$ (which again for the sake of simplicity is supposed to be of the form
$B_k=B(0,k)$ without loss of generality). Then there exists a box
$Q_{k-1} \in \mathfrak{B}_{k-1}$ of scale $k-1$ such that every box of scale $k-1$ that intersects $B_k,$ but not $Q_{k-1},$ is good.
Let
\begin{align}\label{eq:defTa}
 \begin{split}
L_{Q_{k-1}}&:=\min\big \{ l \cdot z -N_{k-2}-1 \, : \, z\in Q_{k-1} \big \},\\
 R_{Q_{k-1}}&:=\max \big \{ l \cdot z + N_{k-1}/2 +1 \,: \, z\in Q_{k-1} \big \}
\le L_{Q_{k-1}}+ 3N_{k-1},\\
H^{(Q_{k-1})}&:=\inf \big \{n \in \N_0: X_n \cdot l \in [L_{Q_{k-1}}, R_{Q_{k-1}}] \big \} \leq\infty.
\end{split}
\end{align}

As alluded to, we will make use of a one-dimensional random walk with drift which, at every unit of time, being at
$x \in \Z \backslash [L_{Q_{k-1}},R_{Q_{k-1}}],$
moves $N_{k-2}$ steps
to the right with probability ${1-e^{-c_{k-1}''N_{k-1}}}$ and $N_{k-1}$
steps to the left with probability ${e^{-c_{k-1}''N_{k-1}}.}$ 
For $a, b \in \R$ with $a < b$ we use the notation $[[a,b]] := [a,b] \cap \Z.$
Then from any
$x \in [[L_{Q_{k-1}},R_{Q_{k-1}}]],$ it jumps $N_{k-2}$
steps to the right with probability $\kappa^{c_1 N_{k-2}}$ and $N_{k-1}$ steps to the left
 with probability ${1-\kappa^{c_1 N_{k-2}}.}$ 
Denote such a walk by $(Z_n)$ and by $P_y$ the corresponding probability measure conditional on ${\{Z_0 = y\}.}$

We start with proving the estimates
\begin{align} \label{eq:befta}
\begin{split}
\sup_{x \in \widetilde{B}_k} &P_{x,\omega}\big( H^{(Q_{k-1})} < H_{\partial_+B_k} \wedge H_{\partial_s B_k} \big)\\
&\leq \exp\{ -c_k'N_k \} + 3\exp \Big \{ -\Big (c_{k-1}''- \frac{\ln 2 }{N_{k-1}} \Big) (N_k-2N_{k-1}-R_{Q_{k-1}}) \Big\}
\end{split}
\end{align}
and
\begin{equation}\label{eq:afta}
\sup_{y \in [[L_{Q_{k-1}}, R_{Q_{k-1}}]]}
P_{y} \big( H^{-e_1}_{N_k/2}(Z) < H^{e_1}_{N_k}(Z) \big) \leq 6\kappa^{-3c_1 N_{k-1}}
\big( 
\exp\{ -c_{k-1}''N_{k-1} \} \big)^{\frac{L_{Q_{k-1}}+N_k/2}{N_{k-1}}},
\end{equation}
from the combination of which we will be able to deduce \eqref{eq:backSideExitEst}.

To see \eqref{eq:befta}, observe that
the left-hand side of \eqref{eq:befta} can be estimated from above by
\begin{equation} \label{eq:exitProbDecompEst}
 \sup_{x \in \widetilde{B}_k} P_{x,\omega} (H_{\partial B_k} = H_{\partial_s B_k})
+ \sup_{x \in \widetilde{B}_k} P_{x,\omega} \Big( H^{(Q_{k-1})} < H_{\partial_+ B_k}, H_{\partial B_k} \ne H_{\partial_s B_k} \Big).
\end{equation}
The first probability can be estimated from above by $\exp\{ -c'N_{k} \}$ using \eqref{eq:quenchedSideExitEst}.

Note that on the event in the second probability,
up to time 
$H^{(Q_{k-1})}$ the random walk $(X_n)$ (and hence $(Y_n)$) only visits good boxes of scale $k-1$. 
Therefore, using the induction hypothesis \eqref{eq:backSideExitEst} in combination with a comparison of the exit probabilities for $Y_\cdot \cdot l$ with
those for $Z_\cdot,$
we get that with \eqref{eq:exitProbDecompEst}, the left-hand side of \eqref{eq:befta} can be estimated from above by
\begin{align*}
\sup_{x \in \widetilde{B}_k} &P_{x,\omega}
\big (H^{(Q_{k-1})} 
< H_{\partial_+B_k} \wedge H_{\partial_s B_k} \big)\\
&\leq \exp\{ -c'_kN_k \} + \sup_{x \in \widetilde{B}_k} P_{\lfloor x\cdot l \rfloor} \Big (H_{R_{Q_{k-1}}}^{-e_1}(Z) < H_{N_k}^{e_1} (Z) \Big)\\
 &\leq \exp\{ -c'_kN_k \} 
+ 3\sup_{x \in \widetilde{B}_k} \big( \exp \big \{-c_{k-1}'' N_{k-1} + \ln 2 \big\} \big)^{ \frac { \lfloor x \cdot l \rfloor -(R_{Q_{k-1}}+ N_{k-1})}{N_{k-1}}}\\
&\leq \exp\{ -c_k'N_k \} + 3\exp \Big \{ -\Big (c_{k-1}''- \frac{\ln 2 }{N_{k-1}} \Big) (N_k-2N_{k-1}-R_{Q_{k-1}}) \Big\},
\end{align*}
for $N_0$ as in the assumptions, and where the penultimate inequality follows from one-dimensional random walk calculations.
Hence, \eqref{eq:befta} follows.

To see \eqref{eq:afta}, let $y \in [[L_{Q_{k-1}}, R_{Q_{k-1}}]]$ 
and define the events
\begin{equation*}
D^+:=\big \{H_{N_k}(Z) <H_y\circ \theta_1 (Z)
\big\}
\quad \text{ and } \quad
D^-:=\big\{H_{-N_k/2}(Z) <H_y\circ \theta_1(Z)
\big\},
\end{equation*}
with $\theta$ as defined in \eqref{eq:shiftDef}.
Observing that
\begin{align} \label{eq:geomProbDec}
\begin{split}
\sup_{y \in [[L_{Q_{k-1}}, R_{Q_{k-1}}]]}
P_{y} \Big( H^{-e_1}_{N_k/2}(Z) < H^{e_1}_{N_k}(Z) \Big)
&\le \sup_{y \in [[L_{Q_{k-1}}, R_{Q_{k-1}}]]}
\frac{P_y(D^-)}{P_y(D^+) + P_y(D^-)}\\
&\leq \sup_{y \in [[L_{Q_{k-1}}, R_{Q_{k-1}}]]} \frac{P_y(D^-)}{P_y(D^+)},
\end{split}
\end{align}
it will be useful to estimate the probabilities of the events $D^+$ and $D^-.$
Bearing in mind \eqref{eq:defTa} and using assumption \ref{item:standardAssumptions1} we obtain the upper bound
\begin{eqnarray}\label{eq:dplus}
P_{y}(D^+)\geq \big(1-\exp\{ -c_{k-1}''N_{k-1} \}  \big)^{3N_k/(2N_{k-2})}\kappa^{3c_1 N_{k-1}} \geq \frac12 \kappa^{3c_1 N_{k-1}},
\end{eqnarray}
while the strong Markov property in combination with one-dimensional random walk calculations
supplies us with
\begin{eqnarray}\label{eq:dminus}
P_{y}(D^-)\leq 3\big(  \exp\{ -c_{k-1}''N_{k-1} \} \big)^{\frac{L_{Q_{k-1}}+N_k/2}{N_{k-1}}}.
\end{eqnarray}
Plugging \eqref{eq:dplus} and \eqref{eq:dminus} into \eqref{eq:geomProbDec}, display
\eqref{eq:afta} follows.

Noting that for $x\in \widetilde B_k,$ on the event $\{ H_{\partial B_k}=H_{\partial_- B_k} \}$ we have
$P_{x,\omega}$-a.s. that ${H^{(Q_{k-1})} < H_{\partial_+B_k} \wedge H_{\partial_s B_k},}$ 
 we can now apply the strong Markov property
and \eqref{eq:afta} as well as \eqref{eq:befta} to obtain
\begin{align*}
\sup_{x\in \widetilde B_k} &P_{x,\omega}(H_{\partial B_k}=H_{\partial_- B_k})\\
&\leq
\sup_{x \in \widetilde{B}_k} P_{x,\omega}\big( H^{(Q_{k-1})} < H_{\partial_+B_k} \wedge H_{\partial_s B_k} \big)
\times
\sup_{y \in [[L_{Q_{k-1}}, R_{Q_{k-1}}]]}
P_{y} \big (H_{-N_k/2}(Z) < H_{N_k}(Z) \big)\\
&\leq \Big(\exp\{ -c_k'N_k \} + 3\exp \Big \{ -\Big (c_{k-1}''- \frac{\ln 2 }{N_{k-1}} \Big) (N_k-2N_{k-1}-R_{Q_{k-1}}) \Big\}\Big)\\
&\quad \times 6\kappa^{-3c_1 N_{k-1}} \big( 3 \exp \{ -c_{k-1}'' N_{k-1} \} \big)^{\frac{L_{Q_{k-1}}+N_k/2}{N_{k-1}}}\\
&\leq  \exp \Big\{ -N_k \Big( c_{k-1}'' -4\frac{\ln 3}{N_{k-1}} -
\frac{2N_{k-1}  + 3 N_{k-1}\ln 3  + 2 \ln 6 - 3c_1 N_{k-1} \ln \kappa}{N_k} \Big) \Big\},
\end{align*}
where we used $c_{k-1}'' \le 1,$
and \eqref{eq:backSideExitEst} follows for $k.$
\end{proof}

\subsection{Proof of Proposition \ref{prop:superPolDecay}} \label{subsec:propProof}
 
\begin{proof} [Proof of Proposition \ref{prop:superPolDecay}]
In order to apply our previous results, for $L> N_0$ given we implicitly define $k_L$ via
$
N_{k_L+1}+1 \geq L > N_{k_L}+1,
$
which provides us with
\begin{equation} \label{eq:kApprox}
k_L \sim \frac{\ln L}{\ln \ln L}, \quad \text{ as } L \to \infty.
\end{equation}
Furthermore, define the strip-like set
\begin{align*}
\mathcal{S}_L^l := \Big \{ x\in \Z^d : &-N_{k_L} \leq x \cdot l \leq 11N_{k_L+1},\\
&\text{and } \vert x \cdot l_j \vert \leq 3000 N_{k_L}^{3} (N_0+k_L-1)^2(N_0 + k_L)^2 \; \; \forall \; 2 \le j \le d \Big \}
\end{align*}
as well as 
\begin{align*}
\partial_+ \mathcal{S}_L^l := \Big \{ x\in \partial \mathcal{S}_L^l : &x \cdot l > 11N_{k_L+1}, \\
&\text{and } \vert x \cdot l_j \vert \leq 3000 N_{k_L}^{3} (N_0+k_L-1)^2(N_0 + k_L)^2 \; \; \forall \; 2 \le j \le d \Big \}
\end{align*}
For $L$ large enough, one has
\begin{align}
P_0 \big(H_{\partial \mathcal D_L^{l}} \ne H_{\partial_+ \mathcal D_L^{l}} \big)
&\leq P_0 \big( H_{\partial \mathcal{S}_{L}^l} \ne H_{\partial_+ \mathcal{S}_{L}^l} \big) \nonumber\\
\begin{split} \label{eq:sumDExitBd}
&\leq P_0 \left ( \text{All }B_{k_L} \in \mathfrak{B}_{k_L} \text{ intersecting } \mathcal{S}_{L}^l \text{ are good},
H_{\partial \mathcal{S}_{L}^l} \ne H_{\partial_+ \mathcal{S}_{L}^l} \right)\\
&\quad + P_0
\left( \text{There exists }B_{k_L} \in \mathfrak{B}_{k_L} \text{ intersecting } \mathcal{S}_{L}^l \text{ that is bad} \right).
\end{split}
\end{align}
For $L$ as above and using Proposition \ref{prop:likelygood}, the second summand of the above we estimate by
\begin{align*}
P_0 \Big( &\text{There exists }B_{k_L} \in \mathfrak{B}_{k_L} \text{ intersecting } \mathcal{S}_{L}^l \text{ that is bad} \Big)\\
&\leq 2 \big \vert \mathcal{L}_{N_{k_L-1}, N^{3}_{k_L}} \cap \mathcal{S}_{L}^l \big \vert \exp\{ -2^{k_L} \}\\
 &\leq 2\cdot 3000^d (N_0+k_L-1)^{3d}(N_0+k_L)^{3d} \exp\{ -2^{k_L} \}\\
& \leq \exp \big\{-L^\frac{(1+o(1))\ln 2}{\ln \ln L}\big\},
\end{align*}
using \eqref{eq:kApprox} in the last line.

We now bound the first summand of \eqref{eq:sumDExitBd}.
For that purpose, note that if the walk leaves ${100 (N_0 +k_L-1)^2 (N_0 +k_L)^2}$ blocks of $\mathfrak B_{k_L}$
consecutively through their
frontal boundary,
 then ${\{H_{\partial \mathcal{S}_{L}^l} = H_{\partial_+ \mathcal{S}_{L}^l} \}}$ occurs.
 Therefore, we can dominate the event
$\{H_{\partial \mathcal{S}_{L}^l} \ne H_{\partial_+ \mathcal{S}_{L}^l} \}$
from above by the event that one of the blocks $B_{k_L}$ of scale $k_L$ the walk encounters is left not through
$\partial_+ B_{k_L}.$ 
To make this formal, for each $k \in \N$ associate to $x \in \Z^d$ an element $\pi_{k}(x) \in \mathcal L_{N_{k-1}, N_{k}^3}$
such that
$x \in \widetilde B(\pi_{k}(x),k).$ Define the sequence of stopping times for $(X_n)$ given by
\begin{align*}
D^{k_L}_0 &:= 0,\\
D^{k_L}_j &:= \left\{
 \begin{array}{ll}
\inf \left\{ m \in \N \, : \, X_{m+ D^{k_L}_{j-1}} \notin  B \Big( \pi_{k_L} \big(X_{D^{k_L}_{j-1}} \big), k \Big) \right\} + D^{k_L}_{j-1},
&\text{for } j \ge 1\mbox{ if } D^{k_L}_{j-1}< \infty ,\\
\infty, &\mbox{otherwise.}
\end{array}
\right.
\end{align*}
Using this terminology and the strong Markov property at times $D^{k_L}_j,$ $j \in \N,$
we can upper bound the first summand of \eqref{eq:sumDExitBd} by
\begin{align*}
 P_0 \Big( &\text{All }B_{k_L} \in \mathfrak{B}_{k_L} \text{ intersecting } \mathcal{S}_{L}^l \text{ are good},
H_{\partial \mathcal{S}_{L}^l} \ne H_{\partial_+ \mathcal{S}_{L}^l} \Big)\\
&\leq \E \Big( P_{x,\omega}
\big( H_{\partial \mathcal{S}_{L}^l} \ne H_{\partial_+ \mathcal{S}_{L}^l} \big),
\text{all }B_{k_L} \in \mathfrak{B}_{k_L} \text{ intersecting } \mathcal{S}_{L}^l \text{ are good} \Big)\\
&\hspace{-1em} \leq \E \Bigg (
P_{0,\omega} \Big( \exists \;1 \le j \le 100 (N_0 +k_L-1)^2 (N_0 +k_L)^2 : 
  X_{D^{k_L}_j} \notin \partial_+ B \Big( \pi_{k_L} \big( X_{D^{k_L}_{j-1}} \big), k_L  \Big),\\
&\quad \text{all }B_{k_L} \in \mathfrak{B}_{k_L} \text{ intersecting } \mathcal{S}_{L}^l \text{ are good}  \Bigg)\\
&\leq 100 (N_0 +k_L-1)^2 (N_0 +k_L)^2 \exp\{ -cN_{k_L} \} \\
& \quad \times \P \big(\text{All }B_{k_L} \in \mathfrak{B}_{k_L} \text{ intersecting } \mathcal{S}_{L}^l \text{ are good} \big)\\
&\leq 100 (N_0 +k_L-1)^2 (N_0 +k_L)^2 \exp \{ -cN_{k_L} \}\\
&\leq \exp \big\{-L^\frac{(1+o(1))\ln 2}{\ln \ln L}\big\},
\end{align*}
where to obtain the second inequality we took advantage of Proposition \ref{prop:GoodBoxRightExit}.
This finishes the proof.
\end{proof}

\section{Proof of Proposition \ref{prop:quenchedExit} (Atypical quenched exit estimates)} \label{sec:proofPropQuenchedExit}
\begin{proof}[Proof of Proposition \ref{prop:quenchedExit}]
Let $l$ be as in the assumptions of Proposition \ref{prop:quenchedExit}
and $l_1, l_2, \ldots, l_d$ as below display \eqref{eq:piBot}.
Recall \eqref{eq:latticeType} and for each $n \in\mathbb N$ define
$$
\mathbb L_n := \mathcal L_{n, \frac{n^3\ln \ln n}{\ln n}}.
$$
In addition, for
each $x\in\mathbb L_n$ define the parallelograms
$$
\mathcal R_n(x):= \Big \{y\in\Z^d: -2n < (y-x)\cdot l < 2n, 
\vert (y-x) \cdot l_j \vert < 2\frac{n^3 \ln \ln n}{\ln n} \; \; \forall \; 2 \le j \le d \Big\},
$$
and their corresponding central parts
$$
\widetilde {\mathcal R}_n(x):=\Big \{y\in\mathbb Z^d: -n -1< (y-x)\cdot l < n+1, 
\vert (y-x) \cdot l_j \vert < \frac{n^3 \ln \ln n}{\ln n} +1 \; \; \forall \; 2 \le j \le d \Big\},
$$
as well as their frontal boundary parts
$$
\partial_+ \mathcal R_n(x) := \Big\{ y \in \partial \mathcal R_n(x) : (y-x) \cdot l \ge 2n,
\vert (y-x) \cdot l_j \vert < \frac{n^3 \ln \ln n}{\ln n} +1 \; \; \forall \; 2 \le j \le d \Big\},
$$
We chose the term ``parallelogram'' in order for the reader to
be able to distinguish this
setting more easily from that of the boxes in Section \ref{sec:renormalisation};
we do use the notation $\mathcal R_n,$ however, in order to better distinguish from the condition $\Pbox$ for which we already use the letter
$\mathcal P.$
We will denote by $\mathfrak P_n$ the set of parallelograms $\{\mathcal R_n(x):
x\in\mathbb L_n \}$. Denote by $J_{L,n}$ the number of parallelograms
in $\mathfrak P_n$ that intersect $B,$ i.e.,
$$
J_{L,n}:= \big \vert \{\mathcal R_n(x) \in \mathfrak P_n:\mathcal R_n(x) \cap B \ne\emptyset \} \big \vert.
$$
Due to Proposition \ref{prop:superPolDecay}, we obtain
\begin{equation} \label{eq:parExit}
 \sup_{y \in \widetilde {\mathcal R}_n(0)} P_y \big( H_{\partial \mathcal R_n(0)} \ne H_{\partial_+ \mathcal R_n(0)} \big)
 \le \exp \Big\{ - n^{\frac{(1+o(1)) \ln 2}{\ln \ln n}} \Big\},
\end{equation}
as $n \to \infty.$
The next step is to perform a one-step renormalisation involving parallelograms
$\mathcal R_n$ with $n:=\floor{L^{\epsilon(L)}}$.
 A
parallelogram $\mathcal R_n(x)\in \mathfrak P_n$ is defined to be {\it good (with respect to $\omega$)} if
$$
\inf_{y\in\widetilde {\mathcal R}_n(x)} P_{y,\omega} \big( H_{\partial \mathcal  R_n(x)} = H_{\partial_+ \mathcal R_n(x)} \big)
\ge 1-L^{-\epsilon(L)^{-1}}.
$$
Otherwise, $\mathcal R_n(x)$ is defined to be {\it bad (with respect to $\omega$)}. Note now that
by Markov's inequality and the invariance of $\P$ under translations of $\Z^d,$
\begin{align} \label{eq:badnessProbBd}
\begin{split}
\P(\mathcal R_n(x) \text{ is bad})
&=\P \left( \sup_{y\in \widetilde {\mathcal  R}_n(x)} P_{y,\omega} \big( {H_{\mathcal R_n(x)}}\not= H_{\partial_+ \mathcal R_n(x)} \big)
>L^{-\epsilon(L)^{-1}} \right)\\
&\le L^{3d\epsilon(L) +\epsilon(L)^{-1}}\sup_{y\in\widetilde {\mathcal R}_n(0)} 
P_y\big( H_{\mathcal R_n(0)} \ne H_{\partial_+ \mathcal R_n(0)} \big)\\
&\le L^{3d\epsilon(L)+\epsilon(L)^{-1} } \exp \Big \{ -L^{\frac{(1+o(1))\epsilon(L) \ln 2}{\ln\ln L^{\epsilon(L)}}} \Big \}\\
&\le  \exp \Big \{ \big(3d \epsilon(L) + \epsilon(L)^{-1} \big) \ln L - L^{\frac{(1+o(1))\ln 2}{(\ln\ln L)^3}} \Big \},
\end{split}
\end{align}
where the second inequality follows from \eqref{eq:parExit}.
Next, we consider the event $G_{\beta, L} \subset \Omega$ defined via
$$
G_{\beta,L}:=\Big \{\text{the number of  bad  parallelograms in } \mathfrak P_{n}
\text{ that intersect } B \text{ is less than } L^{\beta} \Big\}.
$$
A crude strategy for $X$ starting in $B$ to exit $B$ through $\partial_+ B,$ is to exit all $\mathcal R_n(x)$'s encountered through
$\partial_+ \mathcal R_n(x).$
To make this formal, for each $n \in \N$ associate to $x \in \Z^d$ one of the elements $y \in \mathbb L_n$ 
such that $x \in \widetilde R_n(y),$ and denote this element by $\pi_n(x).$ 
In a fashion reminiscent of the end of Subsection \ref{subsec:propProof}, we
define the sequence of stopping times for $(X_n)$ given by
\begin{align*}
D^n_0 &:= 0,\\
D^n_j &:= \left\{
 \begin{array}{ll}
\inf \big\{ k \in \N \, : \, X_{k+ D^n_{j-1}} \notin \mathcal R_n(\pi_k(X_{D^n_{j-1}})) \big\} + D^n_{j-1},
&\text{for } j \ge 1\mbox{ if } D^n_{j-1}< \infty ,\\
\infty, &\mbox{otherwise.}
\end{array}
\right.
\end{align*}
Note that following the above crude strategy,
the number of
bad parallelograms of type $\mathcal R_n(x)$ encountered by the random walk is at
most $L^{\beta(L)}.$
Thus, using the strong Markov property and
\ref{item:standardAssumptions1}
 we observe that
for $\omega\in G_{\beta(L),L}$, 
\begin{align}\label{eq:behaviorG}
\begin{split}
 P_{0,\omega} \big (H_{\partial B} = H_{\partial_+ B} \big)
&\ge  P_{0,\omega} \left( X_{D^n_j} \in \partial_+ \mathcal R_n \big( \pi_n \big (X_{D^n_{j-1}} \big ) \big), \; 
\forall \;1 \le j \le \Big \lceil \frac{L}{\floor{L^{\epsilon(L)}}} \Big \rceil \right) \\
 &\ge\Big( \exp\{-c_1L^{ \epsilon(L)} \big\} \Big)^{L^{\beta(L)}} \big (1-L^{-\epsilon(L)^{-1}} \big)^L\\
&> \frac{1}{2} \exp \big\{ -c_1 L^{\epsilon(L) + \beta(L)} \big\},
\end{split}
\end{align}
for all $L$ large enough, and
where $c_1$ has been defined in \eqref{eq:c4Def}.
On the other hand, one has that
\begin{equation}
\label{probabilityG}
\P \big( G_{\beta(L), L}^c \big)\le 5^d \frac{e}{ \lceil L^{\beta(L)}/5^d \rceil!}.
\end{equation} 
 Indeed, writing $N_L$ for the number
of bad parallelograms in $\mathfrak P_{n}$ that intersect $B,$
we get ${G_{\beta(L),L}^c = \{N_L \geq L^{\beta(L)}\}.}$ 

We now take advantage of the following claim, the proof of which we omit.
\begin{claim}
 There exist subsets $\mathbb L_n^1, \ldots, \mathbb L_n^{5^d}$ of $\mathbb L_n$ such that
 \begin{itemize}
  \item 
  $$
  \bigcup_{j=1}^{5^d} \mathbb L_n^j = \mathbb L_n,
  $$
  and
  \item
  for all $j \in \{1, \ldots, 5^d\}$ and all $x,y \in \mathbb L_n^j$ one has
  $$
  \mathcal R_n(x) \cap \mathcal R_n(y) = \emptyset.
  $$
   \end{itemize}
\end{claim}
This claim in combination with the assumption \ref{item:IIDassumption}
yields that $N_L$ can be stochastically dominated by $\sum_{j=1}^{5^d}N_L^j,$ where the $N_L^j,$ $1 \le j \le 5^d,$
 are independent identically distributed
binomial random variables (defined on some probability space with probability measure $P$) with parameters 
$J_{L,n}$ and $\P(\mathcal R_n(0) \text{ is bad}),$ where, in particular,
\begin{equation}
\label{lattice}
J_{L,n} \leq C L^{d}
\end{equation}
for some constant $C$ and all $L$.

Next, note that for any binomially distributed random variable with expectation $1$, i.e. of the type $Y_n \sim \text{Bin}(n,n^{-1})$,
we have
for all $n \in \N$ and $0 \le k \le n$ that 
$$
P(Y_n \ge k) \le \frac{e}{k!}.
$$
Indeed, we compute
\begin{align*}
 P(Y_n \ge k)
&\le \sum_{j=k}^n {n \choose j} n^{-j}
\le \sum_{j=k}^n \frac{1}{j!} \le \frac{e}{k!}.
\end{align*}
Now due to \eqref{eq:badnessProbBd} and \eqref{lattice},
for $L$ large enough, the $N_L^j$ are stochastically dominated by binomial random variables of the type 
$Y_n.$ Thus, we obtain that
\begin{align} \label{Binomial}
\begin{split}
\P \big( N_L \geq L^{\beta(L)} \big)
&\le P \left( \sum_{j=1}^{5^d} N_L^j \geq L^{\beta(L)} \right)
\leq 5^d \frac{e}{ \lceil L^{\beta(L)}/5^d \rceil!}
\end{split}
\end{align}
Hence,
inequality (\ref{probabilityG}) follows and
 combining \eqref{eq:behaviorG} with (\ref{probabilityG}), we
finish the proof of Proposition \ref{prop:quenchedExit}.
\end{proof}

        

\appendix
\section{}




 \ack 

A.F. Ram\'irez thanks David Campos for useful discussions.

N. Berger has been partially supported by ERC StG grant 239990,
A. Drewitz by an ETH Fellowship and the Edmund Landau Minerva Center for Research in Mathematical Analysis and Related Areas,
and A.F. Ram\'\i rez by Fondo Nacional de Desarrollo Cient\'\i fico
y Tecnol\'ogico grant 1100298.

\frenchspacing
\bibliographystyle{plain}

\begin{thebibliography}{99}

\bibitem[Ber12]{Be-09}
Noam Berger.
\newblock Slowdown estimates for ballistic random walk in random environment.
\newblock {\em J. Eur. Math. Soc.}, 14:127--174, 2012.

\bibitem[BK91]{BrKu-91}
J.~Bricmont and A.~Kupiainen.
\newblock Random walks in asymmetric random environments.
\newblock {\em Comm. Math. Phys.}, 142(2):345--420, 1991.

\bibitem[Che67]{Ch-67}
A.~A. Chernov.
\newblock Replication of a multicomponent chain by the ``lightning mechanism''.
\newblock {\em Biophysics}, 12:no. 2, 336--341, 1967.

\bibitem[DR11]{DrRa-09b}
Alexander Drewitz and Alejandro~F. Ram{\'{\i}}rez.
\newblock Ballisticity conditions for random walk in random environment.
\newblock {\em Probab. Theory Related Fields}, 150(1-2):61--75, 2011.

\bibitem[DR12]{DrRa-10}
Alexander Drewitz and Alejandro~F. Ram{\'{\i}}rez.
\newblock Quenched exit estimates and ballisticity conditions for
  higher-dimensional random walk in random environment.
\newblock {\em Ann. Probab.}, 40(2):459--534, 2012.

\bibitem[Kal81]{Ka-81}
Steven~A. Kalikow.
\newblock Generalized random walk in a random environment.
\newblock {\em Ann. Probab.}, 9(5):753--768, 1981.

\bibitem[KKS75]{KeKoSp-75}
H.~Kesten, M.~V. Kozlov, and F.~Spitzer.
\newblock A limit law for random walk in a random environment.
\newblock {\em Compositio Math.}, 30:145--168, 1975.

\bibitem[MPRW83]{MaPaRuWi-83}
E.~Marinari, G.~Parisi, D.~Ruelle, and P.~Windey.
\newblock Random walk in a random environment and $1/f$ noise.
\newblock {\em Phys. Rev. Lett.}, 50(17):1223--1225, Apr 1983.


\bibitem[Sin82a]{Si-82}
Ya.~G. Sina{\u\i}.
\newblock The limit behavior of a one-dimensional random walk in a random
  environment.
\newblock {\em Teor. Veroyatnost. i Primenen.}, 27(2):247--258, 1982.

\bibitem[Sin82b]{Si-82b}
Ya.~G. Sina{\u\i}.
\newblock Lorentz gas and random walks.
\newblock In R.~Schrader, R.~Seiler, and D.~Uhlenbrock, editors, {\em
  Mathematical Problems in Theoretical Physics}, volume 153 of {\em Lecture
  Notes in Physics}, pages 12--14. 1982.

\bibitem[Sol75]{So-75}
Fred Solomon.
\newblock Random walks in a random environment.
\newblock {\em Ann. Probability}, 3:1--31, 1975.

\bibitem[ST11]{SaTo-11}
Christophe Sabot and Laurent Tournier.
\newblock Reversed {D}irichlet environment and directional transience of random
  walks in {D}irichlet environment.
\newblock {\em Ann. Inst. Henri Poincar\'e Probab. Stat.}, 47(1):1--8, 2011.

\bibitem[SZ99]{SzZe-99}
Alain-Sol Sznitman and Martin Zerner.
\newblock A law of large numbers for random walks in random environment.
\newblock {\em Ann. Probab.}, 27(4):1851--1869, 1999.

\bibitem[Szn01]{Sz-01}
Alain-Sol Sznitman.
\newblock On a class of transient random walks in random environment.
\newblock {\em Ann. Probab.}, 29(2):724--765, 2001.

\bibitem[Szn02]{Sz-02}
Alain-Sol Sznitman.
\newblock An effective criterion for ballistic behavior of random walks in
  random environment.
\newblock {\em Probab. Theory Related Fields}, 122(4):509--544, 2002.

\bibitem[Szn04]{Sz-04}
Alain-Sol Sznitman.
\newblock Topics in random walks in random environment.
\newblock In {\em School and {C}onference on {P}robability {T}heory}, ICTP
  Lect. Notes, XVII, pages 203--266 (electronic). Abdus Salam Int. Cent.
  Theoret. Phys., Trieste, 2004.

\bibitem[Tem72]{Te-72}
D.~E. Temkin.
\newblock One-dimensional random walks in a two-component chain.
\newblock {\em Dokl. Akad. Nauk SSSR}, 206:27--30, 1972.

\bibitem[ZM01]{ZeMe-01}
Martin P.~W. Zerner and Franz Merkl.
\newblock A zero-one law for planar random walks in random environment.
\newblock {\em Ann. Probab.}, 29(4):1716--1732, 2001.

\end{thebibliography}

\end{document}